\documentclass[12pt]{article}
\usepackage{fullpage}
\usepackage{amssymb}
\usepackage{amsmath}
\usepackage{amsthm}
\newtheorem{theorem}{Theorem}

\newtheorem{corollary}{Corollary}

\newtheorem{lemma}{Lemma}

\begin{document}

\title{\Large Powerful numbers in $(1^{\ell}+q^{\ell})(2^{\ell}+q^{\ell})\cdots (n^{\ell}+q^{\ell})$}
\author{\large Quan-Hui Yang$^{1}$\footnote{
Emails:~yangquanhui01@163.com,~zhaoqingqing116@163.com.}
~~Qing-Qing Zhao$^{2}$}
\date{} \maketitle
 \vskip -3cm
\begin{center}
\vskip -1cm { \small
1. School of Mathematics and Statistics, Nanjing University of Information \\
Science and Technology, Nanjing 210044, China }
 \end{center}

 \begin{center}
{ \small 2. Wentian College, Hohai University,  Maanshan 243031,
China }
 \end{center}

\begin{abstract} Let $q$ be a positive integer. Recently, Niu and
Liu proved that if $n\ge \max\{q,1198-q\}$, then the product
$(1^3+q^3)(2^3+q^3)\cdots (n^3+q^3)$ is not a powerful number. In
this note, we prove that (i) for any odd prime power $\ell$ and
$n\ge \max\{q,11-q\}$, the product
$(1^{\ell}+q^{\ell})(2^{\ell}+q^{\ell})\cdots (n^{\ell}+q^{\ell})$
is not a powerful number; (2) for any positive odd integer $\ell$,
there exists an integer $N_{q,\ell}$ such that for any positive
integer $n\ge N_{q,\ell}$, the product
$(1^{\ell}+q^{\ell})(2^{\ell}+q^{\ell})\cdots (n^{\ell}+q^{\ell})$
is not a powerful number.



\end{abstract}

\section{Introduction}

A positive integer $t$ is called a powerful number, if $t>1$ and
$p^2\mid t$ for every prime divisor $p$ of $t$ (see
\cite{Golomb}). In 2008, Cilleruelo \cite{Cilleruelo} proved that,
for any integer $n>3$, the product $(1^2+1)(2^2+1)\cdots (n^2+1)$
is not a square. Amdeberhan, Medina and Moll \cite{Amdeberhan}
claimed that if $n>12$ and $\ell$ is an odd prime, then
$(1^{\ell}+1)(2^{\ell}+1)\cdots (n^{\ell}+1)$ is not a square.
G\"{u}rel and Kisisel \cite{GurelJNT} confirmed the claim for
$\ell=3$, while Zhang and Wang \cite{zhang} confirmed the claim
for any prime $\ell\ge 5$. In fact, they proved that
$(1^{\ell}+1)(2^{\ell}+1)\cdots (n^{\ell}+1)$ is not a powerful
number. Later, Chen et al. \cite{chen2,{chen}} proved that if
$\ell$ is an odd integer with at most two distinct prime factors,
then $(1^{\ell}+1)(2^{\ell}+1)\cdots (n^{\ell}+1)$ is not a
powerful number. There are many related results on this topic, one
can refer to
\cite{CillerueloLuca,Fang,GurelMonth,Gurel,Hong,Spiegelhalter,Yang,zhangFu}.

Recently, Niu and Liu \cite{Niu} extended the work of G\"{u}rel
and Kisisel and proved that the following theorem.

\noindent{\bf Theorem A.} {\em For any positive integers $q$ and
$n\ge \max\{q,1198-q\}$, the product $(1^3+q^3)(2^3+q^3)\cdots
(n^3+q^3)$ is not a powerful number.}

In this paper, we generalize the results of Niu and Liu in the
following theorem.

\begin{theorem}\label{mainthm} Let $q$ be a positive integer and $\ell$
be an odd prime power. For any integer $n\ge \max\{q,11-q\}$, the
product $(1^{\ell}+q^{\ell})(2^{\ell}+q^{\ell})\cdots
(n^{\ell}+q^{\ell})$ is not a powerful number.
\end{theorem}

The next theorem is a generalization of Theorem 2 in \cite{chen}.

\begin{theorem}\label{mainthm2} For any positive integer $q$ and odd positive integer $\ell$, there exists an integer
$N_{q,\ell}$ such that for any positive integer $n\ge N_{q,\ell}$,
the product $(1^{\ell}+q^{\ell})(2^{\ell}+q^{\ell})\cdots
(n^{\ell}+q^{\ell})$ is not a powerful number.
\end{theorem}

\section{Preliminary lemmas}

\begin{lemma}\label{lem1} Let $p$ be a prime and $q,\ell$ be positive integers with $2\nmid \ell$ and $\gcd(\ell,p-1)=1$.
Then the congruence equation $x^{\ell}+q^{\ell}\equiv
0~(\text{mod}~p)$ has only one solution $x\equiv
-q~(\text{mod}~p)$.
\end{lemma}

\begin{proof} If $p\mid q$, then the congruence equation has only one
solution $x\equiv 0\equiv -q~(\text{mod}~p)$, the result is true.
Now we assume $p\nmid q$. Let $g$ be a primitive root modulo $p$.
Then $g^{\frac{p-1}{2}}\equiv -1~(\text{mod}~p)$. Let $x\equiv
g^t~(\text{mod}~p)$, $q\equiv g^m~(\text{mod}~p)$, where $0\le
t,m\le p-2$. Then the congruence equation $x^{\ell}+q^{\ell}\equiv
0~(\text{mod}~p)$ is equivalent to $g^{t\ell}\equiv
g^{\frac{p-1}{2}+m\ell}~(\text{mod}~p)$, that is, $\ell t\equiv
\frac{p-1}{2}+m\ell~(\text{mod}~p-1)$. Since $(\ell,p-1)=1$, it
follows that $t$ has only one solution. Hence $x$ also has only
one solution. By $2\nmid \ell$, it is easy to see that $x\equiv
-q~(\text{mod}~p)$ is the only solution.
\end{proof}

\begin{corollary}\label{cor1} Let $q$ be a positive integer and $\ell=k^s$, where $k$ is an odd prime and $s$ is a positive integer.
If $p$ is a prime with $k\nmid p-1$, then the congruence equation
$x^{\ell}+q^{\ell}\equiv 0~(\text{mod}~p)$ has only one solution
$x\equiv -q~(\text{mod}~p)$.
\end{corollary}

For a nonzero integer $m$ and a prime $p$, let $\nu_p(m)$ denote
the smallest nonnegative integer $k$ such that $p^k\mid m$ and
$p^{k+1}\nmid m$.

\begin{lemma}\label{lem2} Let $\ell=k^s$ be an odd prime power, $p$ be
a prime and $q$ be a positive integer such that $p>q$, $p\not=k$
and $k\nmid p-1$. If $p-q\le n\le 2p-q-1$, then the product
$(1^{\ell}+q^{\ell})(2^{\ell}+q^{\ell})\cdots (n^{\ell}+q^{\ell})$
is not a powerful number.
\end{lemma}
\begin{proof} By Corollary \ref{cor1}, the smallest two positive
integers $x$ satisfying $x^{\ell}+q^{\ell}\equiv 0~(\text{mod}~p)$
are $p-q$ and $2p-q$. Noting that $p>q$ and $p\not=k$, we have
$p^2\nmid ((p-q)^{\ell}+q^{\ell})$. Hence, if $p-q\le n\le
2p-q-1$, then $$\nu_p((1^{\ell}+q^{\ell})(2^{\ell}+q^{\ell})\cdots
(n^{\ell}+q^{\ell}))=\nu_p((p-q)^{\ell}+q^{\ell})=1,$$ and so the
product $(1^{\ell}+q^{\ell})(2^{\ell}+q^{\ell})\cdots
(n^{\ell}+q^{\ell})$ is not a powerful number.
\end{proof}

For any positive integers $m$ and $k$, let
$$P(m)=\{p:~p~\text{is a prime},~\frac{m+1}{2}<p\le m+1\},$$
$$P(m;k,1)=\{p:~p~\text{is a prime},~\frac{m+1}{2}<p\le
m+1,~p\equiv 1~(\text{mod}~k)\}.$$

\begin{lemma}(See \cite[Lemma 2.3]{zhang}.)\label{lem3} If $m\not=1,3,5$ or $9$, then $|P(m)|\ge 2.$
\end{lemma}

\begin{lemma}(See \cite[Lemma 2.4]{zhang}.)\label{lem4} If $m\ge 4k$, where $k$ is an odd prime
with $k\ge 5$, then $|P(m)|>|P(m;k,1)|.$
\end{lemma}

\begin{lemma}(See \cite[Lemma 2]{chen}.)\label{lem5} Let $m$ be an
integer with $m\ge 4$ and $m\not=9$. Then there is always an odd
prime $p\in P(m)$ with $p\equiv 2~(\text{mod}~3)$.
\end{lemma}

The following lemma is a powerful lemma for solving exponential
Diophantine equations. It is pretty well-known in the Olympiad
folklore (see, e.g., \cite{LET}) though its origins are hard to
trace.

\begin{lemma}(Lifting The Exponent Lemma.)\label{lem6} Let $x,y$ be
two integers, $\ell$ be an odd positive integer, and $p$ be an odd
prime such that $p\mid x+y$ and none of $x$ and $y$ is divisible
by $p$. We have
$$\nu_p(x^{\ell}+y^{\ell})=\nu_p(x+y)+\nu_p(\ell).$$
\end{lemma}

\section{Proofs of Theorems 1 and 2}

\begin{proof}[Proof of Theorem \ref{mainthm}.] By Lemma
\ref{lem2}, it is enough to prove that there exists a prime $p>q$
with $p\not=k$ and $k\nmid p-1$ such that $p-q\le n\le 2p-q-1$. It
is easy to see that $p-q\le n\le 2p-q-1$ is equivalent to
$\frac{n+q}{2}<p\le n+q$. Since $n\ge q$, it follows that
$p>\frac{n+q}{2}\ge q$. Hence we need to prove that there exists a
prime $p\not=k$ with $p\not\equiv 1~(\text{mod}~k)$ such that
$\frac{n+q}{2}<p\le n+q$.

By $n\ge 11-q$, we have $n+q-1\ge 10$. Hence, by Lemma \ref{lem3},
we obtain \begin{eqnarray}\label{eq1}|P(n+q-1)|\ge
2.\end{eqnarray}

Suppose that $k=3$. Since $n+q-1\ge 10$, by Lemma \ref{lem5},
there exists an odd prime $p$ with $p\equiv 2~(\text{mod}~3)$ such
that $\frac{n+q}{2}<p\le n+q$. It is clear that $p\not=3$.

Now we assume $k\ge 5$.

Case 1. $n<2k-q+1$. If $p\in P(n+q-1;k,1)$, then $p\equiv
1~(\text{mod}~k)$ and $p\ge 2k+1>n+q$, a contradiction. Hence
$|P(n+q-1;k,1)|=0$ in this case. Therefore, by \eqref{eq1}, there
exists at least one prime $p\not=k$ with $p\not\equiv
1~(\text{mod}~k)$ such that $\frac{n+q}{2}<p\le n+q$.

Case 2. $2k-q+1\le n<4k-q+1$. Suppose that
$|P(n+q-1;k,1)|=|P(n+q-1)|$. Then $|P(n+q-1;k,1)|\ge 2$. Hence,
there exist two primes $p_1$ and $p_2$ satisfying $p_1<p_2\le
n+q<4k+1$ and $p_1\equiv p_2\equiv 1~(\text{mod}~k)$. It follows
that $p_1\ge 2k+1$ and $p_2\ge 4k+1$, a contradiction. Hence
$|P(n+q-1)|>|P(n+q-1;k,1)|$. Therefore, there exists a prime $p$
with $p\not\equiv 1~(\text{mod}~k)$ such that $\frac{n+q}{2}<p\le
n+q$. Clearly, $p>\frac{n+q}{2}\ge \frac{2k+1}{2}>k$.

Case 3. $n\ge 4k-q+1$. It follows that $n+q-1\ge 4k$. By Lemma
\ref{lem4}, there exists a prime $p$ with $p\not\equiv
1~(\text{mod}~k)$ such that $\frac{n+q}{2}<p\le n+q$. Clearly
$p>\frac{n+q}{2}\ge \frac{4k+1}{2}>k$.

By three cases above, there exists a prime $p\not=k$ with
$p\not\equiv 1~(\text{mod}~k)$ such that $\frac{n+q}{2}<p\le n+q$.

Therefore, the product
$(1^{\ell}+q^{\ell})(2^{\ell}+q^{\ell})\cdots (n^{\ell}+q^{\ell})$
is not a powerful number.
\end{proof}

\begin{proof}[Proof of Theorem \ref{mainthm2}.] By Dirichlet's
theorem on arithmetic progressions (see \cite[p. 285]{sandor}),
there exists an integer $N_{q,\ell}>q$ such that for any integer
$n\ge N_{q,\ell}$, there is an odd prime $p\in P(n+q-1)$ with
$p\equiv 2~(\text{mod}~\ell)$. Clearly, $\frac{n+q+1}{2}\le p\le
n+q$ and $\gcd(p-1,\ell)=1$. Suppose that the product
$(1^{\ell}+q^{\ell})(2^{\ell}+q^{\ell})\cdots (n^{\ell}+q^{\ell})$
is a powerful number. Noting that $\frac{n+q+1}{2}\le p\le n+q$
and $n\ge N_{q,\ell}>q$, we have $p\ge \frac{n+q+1}{2}\ge q+1$,
and so $\nu_p(\prod_{a=1}^n (a+q))=1$. Hence, by
$$\prod_{a=1}^n (a^{\ell}+q^{\ell})=\prod_{a=1}^n (a+q)\cdot
\prod_{a=1}^n \frac{a^{\ell}+q^{\ell}}{a+q},$$  it follows that
$p\mid \frac{a^{\ell}+q^{\ell}}{a+q}$ for some $1\le a\le n$.
Since $p\mid a^{\ell}+q^{\ell}$, $2\nmid \ell$ and
$\gcd(p-1,\ell)=1$, by Lemma \ref{lem1}, we have $p\mid a+q$. On
the other hand, by $p\equiv 2~(\text{mod}~\ell)$ and $p\ge q+1$,
we have $p\nmid \ell$ and $p\nmid q$, and so $p\nmid a$. Hence, by
Lemma \ref{lem6}, we have
$$\nu_p(a^{\ell}+q^{\ell})=\nu_p(a+q)+\nu_p(\ell)=\nu_p(a+q).$$
That is, $p\nmid \frac{a^{\ell}+q^{\ell}}{a+q}$, a contradiction.

This completes the proof of Theorem 2.

\end{proof}

\section{Acknowledgement} This work was supported by the National Natural Science Foundation
for Youth of China, Grant No. 11501299, the Natural Science
Foundation of Jiangsu Province, Grant Nos. BK20150889,~15KJB110014
and the Startup Foundation for Introducing Talent of NUIST, Grant
No. 2014r029.

\clearpage
\end{document}